\documentclass[12pt]{article} 
\usepackage[english]{babel}
\usepackage[latin1]{inputenc}
\usepackage{amsmath}
\usepackage{amsthm}
\usepackage{amssymb}
\usepackage{mathrsfs}
\usepackage{tikz-cd}
\usepackage{hyperref}
\usepackage{comment}
\usepackage{caption}
\usepackage{float}

\newtheorem{manualtheoreminner}{Theorem}
\newenvironment{manualtheorem}[1]{%
	\renewcommand\themanualtheoreminner{#1}%
	\manualtheoreminner
}{\endmanualtheoreminner}

\newtheorem{theorem}{Theorem}[section]
\newtheorem{proposition}[theorem]{Proposition}
\newtheorem{corollary}[theorem]{Corollary}
\newtheorem{lemma}[theorem]{Lemma}
\newtheorem{conjecture}[theorem]{Conjecture}
\newtheorem*{claim}{Claim}

\theoremstyle{definition}

\newtheorem{remark}[theorem]{Remark}
\newtheorem{example}[theorem]{Example}
\newtheorem{question}[theorem]{Question}

\DeclareMathOperator{\gon}{\operatorname{gon}}
\DeclareMathOperator{\Div}{\operatorname{Div}}
\DeclareMathOperator{\Proof}{\operatorname{Proof}}
\DeclareMathOperator{\outdeg}{\operatorname{outdeg}}

\newcommand\todo[1]{{\textcolor{red}{#1}}}

\title{Weierstrass sets on finite graphs}
\author{Alessio Borzì}

\begin{document}
	
	\maketitle
	
	\begin{abstract}
		We study two possible tropical analogues of Weierstrass semigroups on graphs, called rank and functional Weierstrass sets. We prove that on simple graphs, the first is contained in the second. We completely characterize the subsets of $\mathbb{N}$ arising as a functional Weierstrass set of some graph. Finally, we give a sufficient condition for a subset of $\mathbb{N}$ to be the rank Weierstrass set of some graph, allowing us to construct examples of rank Weierstrass sets that are not semigroups.
	\end{abstract}
	
	\section{Introduction}
	
	Let $X$ be a smooth projective algebraic curve of genus $g$ and fix a point $P \in X$. Denote by $H(P)$ the set of pole orders at $P$ of rational functions regular on $X \setminus \{P\}$. By the Weierstrass gap theorem (see \cite[III.5.3]{farkas1992riemann}), the set of \emph{gaps} $G(P) = \mathbb{N} \setminus H(P)$ has cardinality exactly $g$. This implies that $H(P)$ is a \emph{numerical semigroup}, that is, a cofinite additive submonoid of $\mathbb{N}$. The numerical semigroups arising in this way are called \emph{Weierstrass semigroups}. We have $G(P) = \{ 1,\dots,g \}$ except in a finite number of points, called \emph{Weierstrass points} of $X$ (see \cite[III.5.9]{farkas1992riemann}).
	
	In 1893 Hurwitz \cite{hurwitz1892algebraische} asks if all the numerical semigroups arise in this manner. Several years later, in 1980, Buchweitz \cite{buchweitz1980zariski} showed that the numerical semigroup $S = \langle 13, 14, 15, 16, 17, 18, 20, 22, 23 \rangle$ is not Weierstrass (see also \cite[page 499]{eisenbud1987weierstrass}). The proof essentially gives the following necessary condition for a semigroup to be Weierstrass: the $m$-sumset of the set of gaps must satisfy $|mG(P)| \leq (2m-1)(g-1)$ for any integer $m  \geq 2$. Several numerical semigroups not satisfying the previous condition are constructed in \cite{komeda1998nonweierstrass}. Furthermore, in \cite{eliahou2020buchweitz} it was proved that for a fixed numerical semigroup $S$, the set of integers $m$ that do not satisfy the above condition is finite. Despite these results, little is known more generally about the family of Weierstrass semigroups. For instance, the problem of determining its	density in the set of all numerical semigroups is still open \cite{kaplan2013proportion}.
	
	After the advent of tropical geometry, the tropical analogues of many classical results in algebraic geometry were found. Baker and Norine \cite{baker2007riemann} proved a Riemann-Roch theorem for graphs, which was successively extended by Gathmann and Kerber \cite{gathmann2008riemann} and Mikhalkin and Zharkov \cite{mikhalkin2008jacobian} to \emph{metric graphs}, namely (abstract) tropical curves. The analogue notion of Weierstrass points on graphs was studied for instance in \cite[Section 4]{baker2008specialization} and \cite{richman2019distribution}.
	
	Inspired by a work of Kang, Matthews and Peachey \cite{kang2019laplacian}, in this paper we investigate possible tropical analogues of Weierstrass semigroups. We will focus our attention on graphs rather than metric graphs, the latter are left for future work. It was already noted in \cite{kang2019laplacian} that two possible non-equivalent definitions can be given, as we now explain. Throughout this paper, a \emph{graph} will mean a finite connected multigraph having no loop edges. Let $G$ be a graph and fix a vertex $P \in V(G)$ of $G$. The \emph{functional Weierstrass set} of $G$ at $P$ is defined by
	\[ H_f(P) = \{ n \in \mathbb{N} : \exists f \in \mathcal{M}(G) \text{ that has a unique pole of order $n$ at $P$} \}, \]
	where $\mathcal{M}(G)$ is the set of all integer-valued functions on the vertices of $G$. The \emph{rank Weierstrass set} of $G$ at $P$ is defined by
	\[ H_r(P) = \{ n \in \mathbb{N} : r(nP) > r((n-1)P) \}, \]
	where $r(D)$ denotes the \emph{rank} of the divisor $D$ of the graph $G$, in the sense of Baker and Norine \cite{baker2007riemann} (see Section \ref{section 2}). Classically, for curves, we have $H_f(P) = H_r(P) = H(P)$. However this is not the case for graphs, for instance the cardinality of the set difference $H_f(P) \setminus H_r(P)$ can be arbitrarily large \cite[Proposition 3.9]{kang2019laplacian}.
	
	Our first main result was conjectured in \cite{kang2019laplacian}, and relates the two sets when $G$ is a graph with no multiple edges and more than one vertex, that in this paper will be called \emph{simple}.
	
	\begin{manualtheorem}{A}[Theorem \ref{thm:functional contains rank}]
		Let $G$ be a simple graph. For every $P \in V(G)$ we have $H_r(P) \subseteq H_f(P)$.
	\end{manualtheorem}
	
	As an application of the previous theorem, we calculate the rank and functional Weierstrass set of the graphs $K_{n+1}$ and $K_{n,m}$.
	
	Secondly, we completely characterize the subsets of $\mathbb{N}$ arising as functional Weierstrass sets of graphs and of simple graphs, answering a question in \cite{kang2019laplacian}.
	
	\begin{manualtheorem}{B}[Theorem \ref{thm:functional characterization}]
		The functional Weierstrass sets of graphs (resp. simple graphs) are precisely the additive submonoids of $\mathbb{N}$ (resp. numerical semigroups).
	\end{manualtheorem}
	
	Further, we give a sufficient condition for a subset of $\mathbb{N}$ to be the rank Weierstrass set of a graph.
	
	\begin{manualtheorem}{C}[Theorem \ref{thm:Weierstrass set sequence}]
		Let $e_1 \geq e_2 \geq \dots \geq e_n \geq 0$ be integers and set $s_i = \nolinebreak \sum_{j=1}^i e_j$. There exists a simple graph $G$ with a vertex $P \in V(G)$ such that
		\[ H_r(P) = \{ 0,s_1,s_2,\dots,s_{n-1}\} \cup (s_n + \mathbb{N}). \]
	\end{manualtheorem}

	The previous theorem allows us to construct families of graphs in which the rank Weierstrass set is not a semigroup (see Example \ref{ex:non semigroup}), justifying the name ``Weierstrass \emph{set}".
	
	\section{Preliminaries}\label{section 2}
	
	In this section, we fix our notation and review the basics and some results of Riemann-Roch theory on finite graphs.
	
	In this paper, a \emph{graph} will mean a finite connected multigraph having no loop edges; a \emph{simple graph} will mean a graph with no multiple edges and more than one vertex. Let $G$ be a graph and let $V(G)$ (resp. $E(G)$) denote the set of vertices (resp. edges) of $G$. The set $\Div(G)$ of \emph{divisors} of $G$ is the free abelian group on $V(G)$. We think of a divisor as a formal integer linear combination of the vertices $D = \sum_{P \in V(G)} a_P P \in \Div(G)$ with $a_P \in \mathbb{Z}$. For convenience, we will write $D(P)$ for the coefficient $a_P$ of $P$ in $D$. The \emph{degree} of a divisor $D$ is defined by $\deg(D) = \sum_{P \in V(G)} D(P) \in \mathbb{Z}$. If $D,D' \in \Div(G)$ are two divisors, then $D \geq D'$ if and only if $D(P) \geq D'(P)$ for all $P \in V(G)$. A divisor $D$ is \emph{effective} if $D \geq 0$. The set of effective divisors of degree $d$ is denoted by $\Div_+^d(G)$.
	
	Let $\mathcal{M}(G) = \operatorname{Hom}(V(G),\mathbb{Z})$ be the set of integer-valued functions on the vertices of $G$. For every vertex $P \in V(G)$, define the \emph{indicator function} $f_P \in \mathcal{M}(G)$ by
	\[ f_P(Q) = \begin{cases}
		-1 & Q = P, \\
		0 & Q \neq P.
	\end{cases} \]
	Let $f \in \mathcal{M}(G)$, and denote by $\mathcal{N}(P)$ the neighbourhood of $P \in V(G)$, that is, the subset of vertices of $G$ adjacent to $P$. Define the \emph{Laplacian operator} $\Delta:\mathcal{M}(G) \rightarrow \Div(G)$ by
	\[ \Delta f = \sum_{P \in V(G)} \left( \sum_{Q \in \mathcal{N}(P)} \big( f(P)-f(Q) \big) \right) P. \]
	The divisors of the form $\Delta f$ are \emph{principal}. For convenience we will write $\Delta_P f$ for the coefficient $\Delta f(P)$. If we think of $f$ as a vector, the Laplacian operator can be seen as the multiplication of the \emph{Laplacian matrix} $Q = D-A$, where $D$ is the diagonal matrix of the degrees of the vertices, and $A$ is the adjacency matrix of $G$. The matrix $Q$ has rank $|V(G)|-1$, and $\ker Q = (1,\dots,1)^t$. From this fact, it is easy to see that every principal divisor has degree $0$.
	
	Two divisors $D,D' \in \Div(G)$ are \emph{linearly equivalent}, written $D \sim D'$, if $D - D' = \Delta f$, for some $f \in \mathcal{M}(G)$. The \emph{linear system} associated to a divisor $D \in \Div(G)$ is
	\[ |D| = \{ E \in \Div(G) : E \sim D, E \geq 0 \}. \]
	The \emph{rank} $r(D)$ of a divisor $D$ is defined as $-1$ if $|D| = \emptyset$, otherwise
	\[ r(D) = \max \{ k \in \mathbb{N} : |D-E| \neq \emptyset,\, \forall E \in \Div_+^k(G) \}. \]
	
	\begin{lemma}\cite[Lemma 2.1]{baker2007riemann}\label{lem:sum rank inequality}
		For all $D_1,D_2 \in \Div(G)$ with $r(D_1),r(D_2) \geq 0$ we have $r(D_1+D_2) \geq r(D_1)+r(D_2)$.
	\end{lemma}
	
	\begin{lemma}\cite[Lemma 2.7]{baker2008specialization}\label{lem:rank minus one}
		Let $G$ be a graph, and let $D \in \Div(G)$. Then $r(D - P ) \geq r(D)-1$ for all
		$P \in V(G)$, and if $r(D) \geq 0$, then $r(D - P ) = r(D) - 1$ for some $P \in V (G)$.
	\end{lemma}
	
	The \emph{canonical divisor} of $G$ is
	\[ K_G = \sum_{P \in V(G)} (\deg(P)-2)P. \]
	It has degree $\deg(K_G) = 2g-2$, where $g = |E(G)| - |V(G)| + 1$ is the \emph{genus} (or \emph{cyclomatic number}) of the graph $G$. We are now ready to state the Riemann-Roch theorem for graphs, proved by Baker and Norine \cite{baker2007riemann}.
	
	\begin{theorem}[Riemann-Roch for graphs]\label{thm:rr graphs}
		Let $D$ be a divisor on a graph $G$ of genus $g$. Then
		\[ r(D) - r(K_G -D) = \deg(D) + 1 - g. \]
	\end{theorem}

	For $A \subseteq V(G)$ and $Q \in A$, let $\outdeg_A(Q)$ denote the number of edges incident with $Q$ and a vertex in $V(G) \setminus A$. Fix $P \in V(G)$. A divisor $D$ is \emph{$P$-reduced} if it is effective in $V(G) \setminus \{P\}$, and every non-empty subset $A \subseteq V(G) \setminus \{P\}$ contains a vertex $Q \in A$ such that $\outdeg_A(Q) > D(Q)$.
	
	\begin{proposition}\cite[Proposition 3.1]{baker2007riemann}
		Let $P$ be a vertex of a graph $G$. For every divisor $D$ in $G$, there exists a unique $P$-reduced divisor $D'$ such that $D \sim D'$.
	\end{proposition}
	
	Following \cite[Section 4]{baker2008specialization}, a vertex $P \in V(G)$ of a graph $G$ of genus $g$, is a \emph{Weierstrass point} if $r(gP) \geq 1$.	We now state an analogue of the Weierstrass gap theorem for graphs.
	
	\begin{lemma}\cite[Lemma 4.2]{baker2008specialization}\label{lem:basics of Weierstrass}
		Let $G$ be a graph of genus $g$, and fix a vertex $P \in V(G)$.
		\begin{enumerate}
			\item $P$ is a Weierstrass point if and only if $\mathbb{N} \setminus H_r(P) \neq \{ 1,\dots,g \}$.
			\item $|\mathbb{N} \setminus H_r(P)| = g$.
			\item $\mathbb{N} \setminus H_r(P) \subseteq \{ 1,2,\dots,2g-1 \}$.
		\end{enumerate}
	\end{lemma}

	Note that, in the classical case for curves, the inclusion $\mathbb{N} \setminus H(P) \subseteq \{ 1,\dots,2g-1 \}$ follows from $|\mathbb{N} \setminus H(P)| = g$ and the fact that $H(P)$ is a semigroup (see \cite[Lemma 2.14]{rosales2009numerical}).
	
	We now describe a binary operation on graphs that we will use frequently in Section \ref{section 4} and \ref{section 5}. Let $G_1$ and $G_2$ be two graphs and $v_1$ and $v_2$ be vertices of respectively $G_1$ and $G_2$. The \emph{vertex gluing} (or \emph{vertex identification}) of $v_1$ and $v_2$ is the graph $G$ obtained from $G_1$ and $G_2$ by identifying $v_1$ and $v_2$ as a new vertex $v$. %We note that vertex gluing is one of the operations involved in Whitney 2-Isomorphism theorem (see \cite[Section 5.3]{oxley1992matroid}).
	
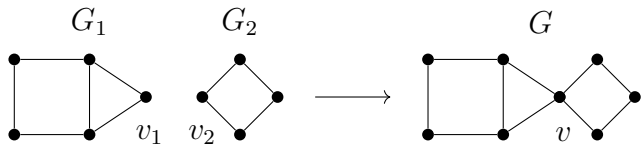
	\begin{figure}[H]
		\centering
		\begin{tikzpicture}
		[scale=1,auto=left,every node/.style={circle,scale=.4,fill=black}]
		
		\node[fill=white, scale=1/.4] (G1) at (1,1.5) {$G_1$};
		\node[fill=white, scale=1/.4] (G2) at (3,1.5) {$G_2$};
		\node[fill=white, scale=1/.4] (G) at (7,1.5) {$G$};
		
		\node (n1) at (0,0) {};
		\node (n2) at (0,1) {};
		\node (n3) at (1,0) {};
		\node (n4) at (1,1) {};
		\node (n5) at (1.75,0.5) {};
		\node[fill=white, scale=1/.4] (v1) at (1.80,0) {$v_1$};
		
		\node (m1) at (2.5,0.5) {};
		\node[fill=white, scale=1/.4] (v2) at (2.5,0) {$v_2$};
		\node (m2) at (3,1) {};
		\node (m3) at (3,0) {};
		\node (m4) at (3.5,0.5) {};
		
		\draw [->] (4,0.5) -- (5,0.5);
		
		\node (N1) at (5.5,0) {};
		\node (N2) at (5.5,1) {};
		\node (N3) at (6.5,0) {};
		\node (N4) at (6.5,1) {};
		\node (N5) at (7.25,0.5) {};
		\node[fill=white, scale=1/.4] (v) at (7.30,0) {$v$};
		
		\node (M1) at (7.25,0.5) {};
		\node (M2) at (7.75,1) {};
		\node (M3) at (7.75,0) {};
		\node (M4) at (8.25,0.5) {};
		
		\foreach \from/\to in {n1/n2,n1/n3,n2/n4,n3/n4,n3/n5,n4/n5,m1/m2,m1/m3,m2/m4,m3/m4,
		N1/N2,N1/N3,N2/N4,N3/N4,N3/N5,N4/N5,M1/M2,M1/M3,M2/M4,M3/M4}
		\draw (\from) -- (\to);
		\end{tikzpicture}
		%\caption{}
	\end{figure}
	
	\section{The inclusion $H_r(P) \subseteq H_f(P)$}
	
	In this section, we will assume that $G$ is a simple graph. We will prove the inclusion $H_r(P) \subseteq H_f(P)$ for every vertex $P \in V(G)$. First, we will need a series of lemmas, inspired by the Cori-Le Borgne algorithm \cite[Proposition 2]{cori2016rank} for the rank of divisors of a complete graph. 
	
	\begin{lemma}\label{lem:P reduced zero}
		Fix a vertex $P \in V(G)$ and let $D$ be a $P$-reduced divisor on $G$. There exists a neighbour $Q \in V(G) \setminus \{P\}$ of $P$ such that $D(Q) = 0$.
	\end{lemma}
	\begin{proof}
		Set $A = V(G) \setminus \{ P \}$ and let $\mathcal{N}(P) \subseteq A$ be the set of neighbours of $P$. Assume by contradiction that $D(Q) \geq 1$ for all $Q \in \mathcal{N}(P)$. Since $G$ is simple, we have $\operatorname{outdeg}_A(Q) = 1$ for all $Q \in \mathcal{N}(P)$. This implies $D(Q) \geq \outdeg_A(Q)$ for all $Q \in A$, contradicting the fact that $D$ is $P$-reduced.
	\end{proof}
	
	Let $D$ be a divisor on $G$ of rank $r$. Using the same terminology as in \cite{cori2016rank, dadderio2018sandpile}, a \emph{proof} for the rank of $D$ is an effective divisor $E$ of degree $r+1$ with $|D-E| = \emptyset$. We denote by $\Proof(D)$ the set of proofs of $D$. Note that if $D \sim D'$, then $\Proof(D) = \Proof(D')$.
	
	%A vertex $P \in V(G)$ is a \emph{proof vertex} for a divisor $D$ if there exists some $A \in \Proof(D)$ such that $A(P) > 0$. Note that for every proof vertex $P$ of a divisor $D$ we have $r(D-P) = r(D)-1$. In fact, if $A \in \Proof(D)$ is such that $A(P) > 0$, then $A-P \in \Proof(D-P)$.
	
	\begin{lemma}\label{lem:proof zero base case}
		Fix a vertex $P \in V(G)$ and let $D$ be a $P$-reduced divisor on $G$ of rank zero. We have $\Proof(D) \setminus \{P\} \neq \emptyset$.
	\end{lemma}
	\begin{proof}
		If $D(P) > 0$, then $P \notin \Proof(D) \neq \emptyset$. Now assume $D(P) = 0$, from Lemma \ref{lem:P reduced zero} there exists a neighbour $Q$ of $P$ such that $D(Q) = 0$. The divisor $D' = D-Q$ is $Q$-reduced. In fact, let $A \subseteq V(G) \setminus \{Q\}$: if $P \notin A$ then, since $D$ is $P$-reduced, we have $\outdeg_A(v) > D(v) = D'(v)$ for some $v \in A$; otherwise if $P \in A$, then $\outdeg_A(P) \geq 1 > 0 = D(P) = D'(P)$. Finally, since $D'(Q) < 0$ and $D'$ is $Q$-reduced, it follows that $Q$ is a proof for $D$ with $Q \neq P$.
	\end{proof}
	
	%The previous proof use the following fact: if D is Q-reduced and D(Q)<0, then D is not linearly equivalent to an effective divisor. Suppose by contradiction that E > 0 is linearly equivalent to D, if E is Q-reducedd we get a contradiction (since there exists a unique Q-reduced divisor up to linear equivalence), otherwise there exists A \subseteq V(G) \setminus Q such that outdeg_A(P) \leq D(P) for every P in A. Therefore, we can "fire the set A". Now we obtain an effective divisor E', and we repeat the procedure for E'. Since at every step we have an effective divisor, but D is not effective, we have a contradiction

	%If D is P-reduced of rank 0, not every vertex with D(v) = 0 is a proof for D. An example is the graph K_{2,3}, with one leaf P attached to a vertex of degree 3. Consider the divisor D such that D(v) = 1 if v=P, 0 otherwise, let Q be the vertex of degree 4, then D(Q) = 0 but Q is not a proof for D
	
	\begin{lemma}\label{lem:proof with no vertex}
		Let $D$ be a divisor on $G$. For every vertex $P \in V(G)$ there exists $E \in \Proof(D)$ such that $E(P) = 0$.
	\end{lemma}
	\begin{proof}
		Without loss of generality, we can assume that $D$ is $P$-reduced. We proceed by induction on the rank $r$ of $D$. The case $r = -1$ is trivial. If $r=0$ the assertion follows from Lemma \ref{lem:proof zero base case}.
		
		Now suppose that $D$ has rank $r \geq 1$ and assume the statement for divisors of rank $r-1$. From Lemma \ref{lem:rank minus one} we have $r(D-P') = r-1$ for some vertex $P' \in V(G)$. By the inductive hypothesis there exists $E' \in \Proof(D-P')$ such that $E'(P) = 0$. Now apply Lemma \ref{lem:proof zero base case} to the $P$-reduced divisor equivalent to $D-E'$. Thus there exists $Q \in \Proof(D-E')$ with $Q \neq P$. We conclude by noting that $E = E'+Q \in \Proof(D)$ and $E(P) = 0$.
	\end{proof}
	
	%Do the previous lemmas give some Le Borgne-type algorithm to compute the rank of a divisor?
	
	Now we prove the main result of the section. We will follow the proof outlined in \cite[Theorem 2.4]{kang2019laplacian} in which the previous Lemma \ref{lem:proof with no vertex} was the key step missing.
	
	\begin{theorem}\label{thm:functional contains rank}
		Let $G$ be a simple graph. For every $P \in V(G)$ we have $H_r(P) \subseteq H_f(P)$.
	\end{theorem}
	\begin{proof}
		Let $n \in H_r(P)$. By Lemma \ref{lem:proof with no vertex}, there exists an effective divisor $E \in \Proof((n-1)P)$ such that $E(P) = 0$. By the choice of $E$ and since $r(nP) > r((n-1)P)$, there exists a function $f \in \mathcal{M}(G)$ such that
		\begin{align*}
			(n-1)P - E + \Delta f \ngeq 0, \\
			nP - E + \Delta f \geq 0.
		\end{align*}
	 	This, together with the fact that $E(P) = 0$, implies that $f$ has a unique pole of order $n$ at $P$, that is $n \in H_f(P)$.
	\end{proof}
	
	\begin{comment}
	
	\begin{lemma}\label{lem:banana graph}
		Let $B_n$ be the graph with two vertices connected by $n$ edges. Then $H_f(P) = n \mathbb{N}$ for every vertex $P$ of $B_n$.
	\end{lemma}
	\begin{proof}
		Let $f \in \mathcal{M}(G)$ with $f(P_1) = a$ and $f(P_2) = b$. Then
		\[ \Delta(f) = n( (a-b)P_1+(b-a)P_2), \]
		therefore $H_f(P_i) = n \mathbb{N}$ for $i = 1,2$.
	\end{proof}
	
	\end{comment}
	
	\begin{remark}\label{rmk:banana graph}
		
		In general, Theorem \ref{thm:functional contains rank} fails when $G$ has just one vertex $P$ (in which case we have $H_f(P) = \{0\}$ and $H_r(P) = \mathbb{N}$) and when $G$ has multiple edges. An example of the last statement is given by the multigraph $B_n$ with two vertices connected by $n$ edges. For every vertex $P \in V(B_n)$, it results $H_f(P) = n \mathbb{N}$ and $H_r(P) = \mathbb{N} \setminus \{ 1,\dots,n-1 \}$, hence $H_r(P) \nsubseteq H_f(P)$.
	\end{remark}
	
	Following the strategy outlined in \cite{kang2019laplacian}, as an application of Theorem \ref{thm:functional contains rank} we calculate the rank Weierstrass set of complete and complete bipartite graphs from their functional Weierstrass set. In fact, in these two cases we have $H_r(P) = H_f(P)$ for every vertex $P$ of the graph.
	
	\begin{lemma}\cite[Porism 2.11]{kang2019laplacian}\label{lem:minimum}
		Let $G$ be a simple graph, let $P \in V(G)$ be a vertex and let $G-P$ be the graph $G$ with the vertex $P$ and its adjacent edges removed. If $G-P$ is connected and $f \in \mathcal{M}(G)$ is a function with a unique pole at $P$, then $f(P) < f(Q)$ for every $Q \in V(G)$.
	\end{lemma}
	
	Let $n \geq 1$ and consider the complete graph $K_{n+1}$.
	
	\begin{lemma}\cite[Proposition 3.7]{kang2019laplacian}\label{lem:functional complete}
		For every vertex $P \in V(K_{n+1})$, we have $H_f(P) = \langle n,n+1 \rangle$.
	\end{lemma}

	\begin{corollary}\label{cor:rank complete}
		For every vertex $P \in V(K_{n+1})$, we have $H_r(P) = \langle n,n+1 \rangle$. 
	\end{corollary}
	\begin{proof}
		By Lemma \ref{lem:functional complete} and Theorem \ref{thm:functional contains rank} we have $H_r(P) \subseteq H_f(P) = \langle n,n+1 \rangle$. Finally, from Lemma \ref{lem:basics of Weierstrass} we have $|\mathbb{N} \setminus H_r(P)| = g(K_{n+1}) = |\mathbb{N} \setminus \langle n,n+1 \rangle|$.
	\end{proof}
	
	Now let $n, m \geq 1$ and consider the complete bipartite graph $K_{n.m}$. The proof of the following lemma is inspired by the proof of \cite[Proposition 3.7]{kang2019laplacian}.
	
	\begin{lemma}\label{lem:functional complete bipartite}
		Let $P \in V(K_{m,n})$ be a vertex of degree $n$, we have
		\[ H_f(P) = n \mathbb{N} \cup (n(m-1) + \mathbb{N}) \]
	\end{lemma}
	\begin{proof}
		If $n$ or $m$ is equal to $1$, then $H_f(P) = \mathbb{N}$, so we assume that $n,m \geq 2$. We label the vertices of $K_{n,m}$ of degree $n$ by $P = P_1,P_2,\dots,P_m$ and the vertices of degree $m$ by $Q = Q_1,\dots,Q_n$. Let $f \in \mathcal{M}(K_{n,m})$ with a unique pole at $P$. By Lemma \ref{lem:minimum} the minimum of $f$ is attained at $P$. Without loss of generality we can assume $f(P) = 0$. Set $f(Q_i) = a + \alpha_i$ for $i \in \{ 1,\dots,n \}$ with $a,\alpha_i \in \mathbb{N}$ and $\alpha_1 = 0$, and $f(P_i) = b+\beta_i$ for $i \in \{ 2,\dots,m \}$ with $b,\beta_i \in \mathbb{N}$ and $\beta_2 = 0$. Now we have
		\begin{align*}
			-\Delta_P f &= na + \sum_{i=2}^n \alpha_i \geq 0, \\
			\Delta_Q f &= a + (m-1)(a-b) - \sum_{i=3}^m \beta_i \geq 0, \\
			\Delta_{P_2} f &= n(b-a) - \sum_{i=2}^n \alpha_i \geq 0.
		\end{align*}
		Now if $a \geq b$, then from the third inequality $0 \geq n(b-a) \geq \sum \alpha_i \geq 0$. Hence $\alpha_i = 0$ for all $i \in \{ 2,\dots,m \}$ and $-\Delta_P f = na \in n\mathbb{N}$. On the other hand, if $a < b$, then from the second inequality
		\[ a+(m-1)(a-b) \geq \sum \beta_i \geq 0 \Rightarrow a \geq (m-1)(b-a) \geq m-1. \]
		This implies $-\Delta_P f = na + \sum \alpha_i \geq n(m-1)$, that is $-\Delta_P f \in n(m-1) + \mathbb{N}$. This proves the inclusion $H_f(P) \subseteq n\mathbb{N} \cup (n(m-1)+\mathbb{N})$.
		
		For the reverse inclusion, it is enough to note that, for the indicator function $f_P$, we have $\Delta f_P = -n P + \sum Q_i$. In addition, for every $t \in \{ 1,\dots,n-1 \}$
		\[ \Delta \left( mf_P + \sum_{i=1}^t f_{Q_i} \right) = -\big( n(m-1)+t \big)P + \sum_{i=2}^m tP_i + \sum_{i=t+1}^n mQ_i. \]
	\end{proof}
	
	Proceeding similarly as in the proof of Corollary \ref{cor:rank complete}, we are able to calculate the rank Weierstrass set of complete bipartite graphs.
	
	\begin{corollary}\label{cor:rank complete bipartite}
		Let $P \in V(K_{m,n})$ be a vertex of degree $n$, we have
		\[ H_r(P) = n \mathbb{N} \cup ((m-1)n + \mathbb{N}) \]
	\end{corollary}
	\begin{comment}
	\begin{proof}
		The result follows from Lemma \ref{lem:functional complete bipartite}, Theorem \ref{thm:functional contains rank} and the equalities $|\mathbb{N} \setminus H_r(P)| = g(K_{n,m}) = \big|\mathbb{N} \setminus \big( n \mathbb{N} \cup ((m-1)n + \mathbb{N}) \big) \big|$, the first given by Lemma \ref{lem:basics of Weierstrass}.
	\end{proof}
	\end{comment}
	
	\begin{remark}
		The computation of the rank Weierstrass set of complete graphs (Corollary \ref{cor:rank complete}) was already implicit in the proof of \cite[Theorem 8]{cools2017gonality}, a result that gives an upper bound for the gonality sequence of complete graphs. In fact, we note that the rank Weierstrass set of a complete graph coincides with its gonality sequence.
	\end{remark}

	\begin{question}
		Under which conditions on the graph $G$ and the vertex $P$ do we have $H_f(P) = H_r(P)$?
	\end{question}
	
	\section{Functional Weierstrass sets}\label{section 4}
	
	In this section, we characterize the subsets of $\mathbb{N}$ that arise as the functional Weierstrass sets of some graph or simple graph.
	
	\begin{comment}
	
	\begin{lemma}
	Suppose that $f \in \mathcal{M}(G)$ has only one pole at $P \in V(G)$. Then $f$ is constant on the connected components of the graph $G'$ obtained from $G$ by deleting $P$ and its incidence edges.
	\end{lemma}
	\begin{proof}
	Let $C$ be a connected component of $G'$ and let $v \in V(C)$ such that
	\[ f(v) = \min\{f(w) : w \in V(C)\}. \]
	
	\end{proof}
	
	\end{comment}
	
	\begin{lemma}\label{lem:functional Weierstrass set structure}
		Let $G$ be a graph and fix a vertex $P \in V(G)$. The functional Weierstrass set $H_f(P)$ is an additive submonoid of $\mathbb{N}$. Further, if $G$ is simple, then $H_f(P)$ is a numerical semigroup.
	\end{lemma}
	\begin{proof}
		We always have $0 \in H_f(P)$. Further, for every $f,g \in \mathcal{M}(G)$, we have $\Delta(f+g) = \Delta f + \Delta g$, so $H_f(P)$ is closed under addition. Moreover, if $G$ is simple, from Theorem \ref{thm:functional contains rank} we have $H_r(P) \subseteq H_f(P)$, and from Lemma \ref{lem:basics of Weierstrass} it follows $|\mathbb{N} \setminus H_f(P)| \leq | \mathbb{N} \setminus H_r(P)| = g(G)$. Therefore $H_f(P)$ is a numerical semigroup.
	\end{proof}
	
	When the graph $G$ is not clear from the context, we denote the functional Weierstrass set by $H_f^G(P)$. For two subsets $A, B \subseteq \mathbb{N}$, we define
	\[ A + B = \{ a+b : a \in A, b \in B \}. \]
	
	\begin{proposition}\label{prop:sum of graphs}
		Let $G_1$ and $G_2$ be two graphs, and let $G$ be the graph obtained from $G_1$ and $G_2$ by the vertex gluing of $P_1 \in V(G_1)$ and $P_2 \in V(G_2)$, and denote by $P \in V(G)$ the identified vertex in $G$. Then
		\[ H^G_f(P) = H^{G_1}_f(P_1) + H^{G_2}_f(P_2). \]
	\end{proposition}
	\begin{proof}
		We will consider $G_1$ and $G_2$ as subgraphs of $G$. For simplicity, set $S = H_f^G(P)$ and $S_i = H^{G_i}_f(P_i)$ for $i \in \{1,2\}$. Let $x \in S_1$, then there exists $f \in \mathcal{M}(G_1)$ such that $\Delta(f) = D - xP_1$ for some effective divisor $D \geq 0$. Consider the extension $f'$ of $f$ to $G$ by setting $f'(v) = f(P)$ for all $v \in V(G_2) \setminus \{P\}$. Then $\Delta(f') = \Delta(f)$ and $x \in S$. This proves $S_1 \subseteq S$. Similarly we obtain $S_2 \subseteq S$, thus $S_1 + S_2 \subseteq S$ since $S$ is closed under addition.
		
		On the other hand, let $x \in S$. Then there exists $f \in \mathcal{M}(G)$ such that $\Delta(f) = D-xP$ for some $D \geq 0$. Substituting $f$ with $f+a$ for some constant $a \in \mathbb{Z}$ if necessary, we can assume that $f(P) = 0$. For $i \in \{1,2\}$, define
		\[ f_i \in \mathcal{M}(G_i) \quad f_i(v) = f(v) \quad \forall v \in V(G_i) \subseteq V(G). \]
		Since $f(P) = 0$, we have $\Delta_v(f_i) = \Delta_v(f) \geq 0$ for all $v \in V(G_i) \setminus \{P\} \subseteq V(G)$, with $i \in \{1,2\}$. Since every principal divisor has degree zero, we have
		\[ \Delta(f_1) = D_1 - x_1 P, \quad \Delta(f_2) = D_2 - x_2 P \]
		for some $x_1,x_2 \in \mathbb{N}$ and some effective divisor $D_i \geq 0$ on $G_i$ for $i \in \{1,2\}$. From the definition we have $f = f_1+f_2$, therefore $\Delta(f) = \Delta(f_1) + \Delta(f_2)$, hence $x = x_1 + x_2 \in S_1 + S_2$.
	\end{proof}

	\begin{corollary}\label{cor:additive submonoid}
		For every additive submonoid $M$ of $\mathbb{N}$ there exists a graph $G$ such that $M = H_f(P)$ for some $P \in V(G)$.
	\end{corollary}
	\begin{proof}
		From \cite[Lemma 2.3]{rosales2009numerical}, every additive submonoid of $\mathbb{N}$ is finitely generated, so suppose that $M = \langle n_1,\dots,n_e \rangle$. Now let $G$ be the graph with vertices $V(G) = \{ P, P_1, \dots, P_e \}$ where the vertex $P$ has $n_i$ edges connected to the vertex $P_i$ for every $i \in \{ 1,\dots,e \}$. From Remark \ref{rmk:banana graph} and Proposition \ref{prop:sum of graphs} it follows that $H_f(P) = n_1 \mathbb{N} + \dots + n_e \mathbb{N} = M$.
	\end{proof}
	
	\begin{corollary}\label{cor:numerical semigroups}
		For every numerical semigroup $S$ there exists a simple graph $G$ such that $S = H_f(P)$ for some $P \in V(G)$.
	\end{corollary}
	\begin{proof}
		Suppose that $S = \langle n_1,\dots,n_e \rangle$. Set $m = \max(\mathbb{N} \setminus S) +2$ and consider the complete bipartite graphs $K_{m,n_1},\dots,K_{m,n_e}$. Fix a vertex of degree $n_i$ in each graph and construct the graph $G$ by identifying these vertices, recursively applying the vertex gluing. Denote with $P$ the identified vertex in $G$. From Proposition \ref{prop:sum of graphs} and Lemma \ref{lem:functional complete bipartite} we obtain $H_f(P) = S$.
	\end{proof}
	
	Using Lemma \ref{lem:functional Weierstrass set structure} and Corollary \ref{cor:additive submonoid} and \ref{cor:numerical semigroups} we now state the main result of this section.
	
	\begin{theorem}\label{thm:functional characterization}
		The functional Weierstrass sets of graphs (resp. simple graphs) are precisely the additive submonoids of $\mathbb{N}$ (resp. numerical semigroups).
	\end{theorem}

	We close the section by calculating the multiplicity of the functional Weierstrass set of a simple graph. Recall that the multiplicity of a numerical semigroup $S$ is the integer $m(S) = \min(S \setminus \{0\})$. Let $G$ be a simple graph and fix a vertex $P \in V(G)$. Denote with $G-P$ the graph obtained from $G$ by removing the vertex $P$ and its adjacent edges.
	
	\begin{lemma}\cite[Theorem 2.10]{kang2019laplacian}\label{lem:multiplicity connected}
		Suppose that $G-P$ is connected. Then $m(H_f(P)) = \deg(P)$.
	\end{lemma}
	
	\begin{proposition}
		Let $G_1,\dots,G_m$ be the connected components of $G-P$, and let $\deg_{G_i}P$ be the number of edges incident with $P$ in $G_i$. Then
		\[ m(H_f(P)) = \min \{ \deg_{G_i}P : i \in \{ 1,\dots,m \} \} \]
	\end{proposition}
	\begin{proof}
		Let $C_i$ be the graph obtained from $G_i$ by adding the vertex $P$ and the edges of $G$ incident with $P$ in $G_i$. The graph $G$ can be seen as the vertex gluing of the graphs $C_i$ along $P$. Now it is enough to apply Proposition \ref{prop:sum of graphs} and Lemma \ref{lem:multiplicity connected}.
	\end{proof}

	\section{Rank Weierstrass sets}\label{section 5}
	
	Let $G$ be a graph and fix a vertex $P \in V(G)$. Define the function $\lambda_P: \mathbb{N} \rightarrow \mathbb{N}$ by
	\[ \lambda_P(k) = \min \{ n \in \mathbb{N} : r(nP) = k \}. \]
	Note that the function $\lambda_P$ is an order preserving bijection between $\mathbb{N}$ and $H_r(P)$. Thus, $\lambda_P$ completely determines $H_r(P)$ and vice versa. We will write $\lambda^G_P$ when the graph $G$ is not clear from the context
	
	\begin{proposition}\label{prop:lambda sum}
		Let $G_1$ and $G_2$ be two graphs and fix $P_i \in V(G_i)$ for $i\in \{1,2\}$. Let $G$ be the vertex gluing of $P_1$ and $P_2$, and let $P$ be the identified vertex. Then
		\[ \lambda^G_P(k) = \max \left\{ \lambda_{P_1}^{G_1}(k_1) + \lambda_{P_2}^{G_2}(k_2) : k_1 + k_2 = k \right\}. \]
	\end{proposition}
	\begin{proof}%[Second proof]
		We will consider $G_1$ and $G_2$ as subgraphs of $G$. First, note that every divisor $E \in \Div_+^k(G)$ can be decomposed as the sum $E = E_1+E_2$ where $E_i \in \Div_+^{k_i}(G_i)$ for $i \in \{1,2\}$ with $k_1+k_2 = k$. Further, if $\Delta f$ is a principal divisor in $G$, without loss of generality we can assume that $f(P) = 0$, so that $f = f_1 + f_2$ with $f_1 = 0$ in $G_2$ and $f_2 = 0$ in $G_1$. It follows that $\Delta f = \Delta f_1 + \Delta f_2$, in other words any principal divisor in $G$ is the sum of two principal divisors in $G_1$ and $G_2$ respectively.
		
		\begin{claim}
			Let $n,k \in \mathbb{N}$, the following statements are equivalent:
			\begin{enumerate}
				\item $|nP-E| \neq \emptyset$ for every $E \in \Div_+^k(G)$,
				\item $n \geq \lambda_{P_1}^{G_1}(k_1) + \lambda_{P_2}^{G_2}(k_2)$ for every $k_1 + k_2 = k$.
			\end{enumerate}
		\end{claim}
		\begin{proof}[Proof of claim]
			First of all, set $n_i = \lambda_{P_i}^{G_i}(k_i)$ for $i \in \{1,2\}$.
			\begin{description}
				\item[$1) \Rightarrow 2)$] Let $k_1,k_2 \in \mathbb{N}$ such that $k_1+k_2 = k$. By the definition of $n_i$, there exists $E_i \in \Div_+^{k_i}(G_i)$ such that $|(n_i-1)P - E_i| = \emptyset$ for $i \in \{1,2\}$. Set $E = E_1+E_2 \in \Div_+^k(G)$. By hypothesis we have
				\[ nP -E + \Delta f = nP + (\Delta f_1 - E_1) + (\Delta f_2 - E_2) \geq 0, \]
				for some $f \in \mathcal{M}(G)$, with $f = f_1 + f_2$ as described above. Assume by contradiction $n < n_1+n_2$, this means that, for some $i \in \{1,2\}$, we have $(n_i-1)P -E_i + \Delta f_i \geq 0$, that is $|(n_i-1)P-E_i| \neq \emptyset$, contradiction.
				
				\item[$2) \Rightarrow 1)$] Let $E \in \Div_+^k(G)$, then $E = E_1 + E_2$ with $E_i \in \Div_+^{k_i}(G_i)$ for $i \in \{1,2\}$ and $k_1+k_2 = k$. By the definition of $n_i$, there exists $f_i \in \mathcal{M}(G_i)$ such that $n_i P - E_i + \Delta f_i \geq 0$ for $i \in \{1,2\}$. Without loss of generality, we can assume that $f_1(P) = f_2(P) = 0$, set $f = f_1 + f_2$, we have
				\[ nP - E + \Delta f \geq \sum_{i \in \{1,2\}} (n_i P - E_i + \Delta f_i) \geq 0 \]
				that is $|nP-E| \neq \emptyset$. \qedhere
			\end{description}
		\end{proof}
		\noindent
		Now write
		\[ \begin{split}
			\lambda_P^G(k)&= \min \{ n \in \mathbb{N} : r(nP) \geq k \} \\
			&= \min \{ n \in \mathbb{N} : |nP-E| \neq \emptyset,\, \forall E \in \Div_+^k(G) \} \\
			&= \min \{ n \in \mathbb{N} : n \geq \lambda_{P_1}^{G_1}(k_1) + \lambda_{P_2}^{G_2}(k_2), \text{ for every }  k_1 + k_2 = k \} \\
			& = \max \left\{ \lambda_{P_1}^{G_1}(k_1) + \lambda_{P_2}^{G_2}(k_2) : k_1 + k_2 = k \right\}. \qedhere
		\end{split} \]
	\end{proof}

	\begin{remark}
		We note that the notion of rank Weierstrass set implicitly appears in \cite[Section 3.2]{pflueger2017chainscycles}. In fact, the so called \emph{Weierstrass partition} of the zero divisor at a marked point $P$ encodes the information of the rank Weierstrass set $H_r(P)$ (see \cite[Definition~3.11]{pflueger2017chainscycles}). Further, we note that one of the main techniques used in \cite{pflueger2017chainscycles} to study the behaviour of Weierstrass partitions is the (analogue of) vertex gluing of an arbitrary metric graph with a cycle.
	\end{remark}

	\begin{theorem}\label{thm:Weierstrass set sequence}
		Let $e_1 \geq e_2 \geq \dots \geq e_n \geq 0$ be integers and set $s_i = \nolinebreak \sum_{j=1}^i e_j$. There exists a simple graph $G$ with a vertex $P \in V(G)$ such that
		\[ H_r(P) = \{ 0,s_1,s_2,\dots,s_{n-1}\} \cup (s_n + \mathbb{N}). \]
	\end{theorem}
	\begin{proof}
		%Denote by $W_{g+1}$ a graph of genus $g$ with a fixed non-Weierstrass point $P_{g+1}$. 
		%and \todo{$G = W_{e_1} + W_{e_2} + \dots + W_{e_n}$ along $P \in V(G)$}
		
		We proceed by induction on $n$. For the base case $n=1$, by Corollary \ref{cor:rank complete bipartite} it is enough to consider the graph $K_{2,e_1}$ and a vertex $P$ of degree $e_1$. Now assume that the theorem is true for $n-1$, and let $G'$ be a graph with a vertex $P_1$ such that
		\[ H_r^{G'}(P_1) = \{ 0, s_1,\dots,s_{n-1}\} \cup (s_{n-1} + \mathbb{N}). \]
		Consider the graph $K_{2,e_n}$ and fix a vertex $P_2$ of degree $e_n$. Let $G$ be the vertex gluing of $P_1$ and $P_2$. From Corollary \ref{cor:rank complete bipartite} we have
		\[ H_r^{K_{2,e_n}}(P_2) = \{ 0 \} \cup (e_n + \mathbb{N}). \]
		Now apply Proposition \ref{prop:lambda sum} to the vertex gluing $G$ of $P_1$ and $P_2$.
	\end{proof}
	
	\begin{remark}
		In the proof of Theorem \ref{thm:Weierstrass set sequence} we glued together complete bipartite graphs $K_{2,e_i}$ along vertices of degree $e_i$. However, we could have used any graph of genus $e_i-1$ with a fixed non-Weierstrass point, i.e. with a fixed vertex in which the Weierstrass set is $\mathbb{N} \setminus \{ 1,\dots,e_i-1 \}$.
	\end{remark}
	
	\begin{remark}
		From Theorem \ref{thm:Weierstrass set sequence} it follows that every Arf numerical semigroup is the rank Weierstrass set of some graph. In fact, it is enough to choose the sequence $e_1 \geq \dots \geq e_n$ to be the multiplicity sequence of the given Arf numerical semigroup. See \cite[Section 2]{barucci2003Arf} for more information about Arf numerical semigroups and their multiplicity sequence.
	\end{remark}
	
	Theorem \ref{thm:Weierstrass set sequence} can be used to construct families of graphs with rank Weierstrass set that is not a semigroup. We now describe an example of such a graph.
	
	\begin{example}\label{ex:non semigroup}
		Let $n = 3$ and $(e_1,e_2,e_3) = (3,2,2)$. Following the idea in the proof of Theorem \ref{thm:Weierstrass set sequence}, we consider the graph $G$ (Figure \ref{ex:graph exmaple 1}) obtained as the vertex gluing of $K_{2,3}$ and two copies of $K_{2,2}$. Let $P \in V(G)$ be the identified vertex of degree $7$. We have
		\[ H_r(P) = \{ 0,3,5,7\} \cup (8 + \mathbb{N}). \]
		Note that $H_r(P)$ is not a semigroup, since $3 \in H_r(P)$, but $3+3 = 6 \notin H_r(P)$.
		\begin{figure}[h!]
 			\centering
 			\begin{tikzpicture}
	 			[scale=1,auto=left,every node/.style={circle,scale=.4,fill=black}]
				
	 			\node (n1) at (0,5) {};
	 			\node (n2) at (0,3) {};
	 			\node (n3) at (-1,4) {};
	 			\node (n4) at (1,4) {};
	 			\node (n5) at (-1,2.5) {};
	 			\node (n6) at (-0.5,2) {};
	 			\node (n7) at (1,2.5) {};
	 			\node (n8) at (0.5,2) {};
	 			\node (n9) at (0,4) {};
				\node (n10) at (1.5,1.5) {};
				\node (n11) at (-1.5,1.5) {};
				
	 			\foreach \from/\to in {n1/n9,n9/n2,n1/n3,n1/n4,n2/n3,n2/n4,n2/n5,n2/n6,n2/n7,n2/n8,n5/n11,n11/n6,n7/n10,n10/n8}
	 			\draw (\from) -- (\to);
 			\end{tikzpicture}
 			\caption{}\label{ex:graph exmaple 1}
 		\end{figure}
	\end{example}

	Theorem \ref{thm:Weierstrass set sequence} gives a sufficient condition for a subset of $\mathbb{N}$ to be the rank Weierstrass set of some graph. We now provide an easy necessary condition.
	
	\begin{proposition}
		Let $G$ be a graph and fix a vertex $P \in V(G)$. For every $n, k \in \mathbb{N}$ we have
		\[ \big|H_r(P) \cap [1,nk]\big| \geq k \big| H_r(P) \cap [1,n] \big|. \]
	\end{proposition}
	\begin{proof}
		From the definition we have $r(nP) = \big| H_r(P) \cap [1,n] \big|$, further from Lemma \ref{lem:sum rank inequality} $r(nkP) \geq k\,r(nP)$.
	\end{proof}

	\begin{question}
		Can we characterize the cofinite subsets $H \subseteq \mathbb{N}$ that are the rank Weierstrass set of some graph? %What is their density in the family of cofinite subets of $\mathbb{N}$?
	\end{question}
	
	In \cite{baker2009harmonic} the notion of harmonic morphism between graphs is studied. It is a discrete analogue of morphisms of curves. In particular, in this context it makes sense to talk about hyperelliptic graphs and double covers.
	
	Classically, we know that a curve $X$ is hyperelliptic if and only if there exists $P \in X$ such that $2 \in H(P)$. An analogous fact is proved in \cite[Theorem A]{torres1994double}: a curve $X$ of genus $g \geq 6 \gamma + 4$ is a double cover of a curve of genus $\gamma \geq 1$ if and only if there exists $P \in X$ such that $H(P)$ has $\gamma$ even gaps.
	
	\begin{question}
		Can we find an analogue of \cite[Theorem A]{torres1994double} for graphs?
	\end{question}
	
	\paragraph{Acknowledgments.} I would like to thank Diane Maclagan, Yoav Len and Victoria Schleis for a careful reading and comments on an earlier version of this manuscript, and Marta Panizzut, Gretchen Matthews, Cong X. Kang and Nathan Pflueger for useful discussions and email exchanges.

	\bibliographystyle{abbrv}
	\bibliography{Reference.bib}
	
 	\noindent
	{\scshape Alessio Borz\`{i}} \quad \texttt{Alessio.Borzi@warwick.ac.uk}\\
 	{\scshape  Mathematics Institute, University of Warwick, Coventry CV4 7AL, United Kingdom.}
	
\end{document}